\documentclass[12pt]{article}
\usepackage{amsmath, amsthm}
\usepackage[top=4cm,left=3cm,right=3cm,bottom=4cm,headheight=13.6pt]{geometry}\usepackage{amsfonts}
\usepackage{amssymb}
\usepackage{fancyhdr}
\usepackage{enumerate}
\usepackage{amsmath,amsthm}
\usepackage[dvips]{graphicx}
\usepackage{color}
\usepackage{eepic}
\usepackage{epic}
\usepackage[dvips]{rotating}
\usepackage{sectsty}
\usepackage{wrapfig}

\theoremstyle{plain}
\newtheorem{theorem}{Theorem}[section]
\newtheorem{lemma}[theorem]{Lemma}

\theoremstyle{definition}

\theoremstyle{remark}

 \numberwithin{equation}{section} 

\begin{document}
\title{ Blow-up  Rate Estimates for Parabolic Equations}

\author{Maan A. Rasheed and Miroslav Chlebik}
\maketitle

\abstract 
We consider the blow-up sets and the upper blow-up rate estimates for two parabolic problems defined in a ball $B_R$ in $R^n;$ firstly, the semilinear heat equation $u_t= \Delta u + e^{u^p}$ subject to the zero Dirichlet boundary conditions, secondly, the problem of the heat equation $u_t= \Delta u$ with the Neumann boundary condition  $\frac{ \partial u}{\partial \eta}=e^{u^p}$ on $\partial B_R\times (0,T),$ where $p>1,$ $\eta$ is the outward normal. 

 \section{Introduction}
 
In this paper, we study two problems of parabolic equations:

\begin{equation}\label{1}\left. \begin{array}{ll}
u_t= \Delta u + e^{u^p},&\quad (x,t) \in  B_R  \times (0,T),\\
u(x,t)=0,&\quad (x,t) \in  \partial B_R  \times (0,T),\\
u(x,0)=u_0(x),& \quad x \in B_R,
\end{array}\right\}  \end{equation}
and
\begin{equation}\label{b1} \left.
\begin{array}{ll}
u_t= \Delta u,& \quad (x,t) \in  B_R  \times (0,T),\\
\frac{ \partial u}{\partial \eta}=e^{u^p},&  \quad (x,t)\in \partial B_R  \times (0,T),\\
u(x,0)=u_0(x),&  \quad x \in B_R,
\end{array} \right\} \end{equation}
 where   $p>1,$ $B_R$ is a ball in $R^n,$ $\eta$ is the outward normal, $u_0$  is  smooth, nonzero, nonnegative, radially symmetric,  moreover, for problem (\ref{1}) it is further required to be nonincreasing radial function, vanishing on $\partial B_R$ and satisfies the following condition
  \begin{equation} \label{6}
  \Delta u_0(x) + e^{u^p_0(x)} \ge 0, \quad  x   \in{B}_R, \end{equation} 
while for problem (\ref{b1}), it is required to satisfy the following conditions
\begin{equation}\label{b33}
  \frac{ \partial u_0}{\partial \eta}=e^{u_0^p},\quad x  \in  \partial B_R,\end{equation} 
\begin{equation}\label{df} \Delta u_0 \ge 0,\quad x\in \overline{B}_R. \end{equation} 

Blow-up phenomena for reaction-diffusion problems in bounded domain have been studied for the first time in \cite{10} by Kaplan, he showed that, if the convex source terms $f=f(u)$ satisfying the condition \begin{equation}\label{T} \int^\infty_U\frac{du}{f(u)} <\infty,\quad U\ge 1,\end{equation} then diffusion cannot prevent blow-up when the initial state is large enough.

The problem of semilinear  parabolic equation defined in a ball, has been introduced in \cite{1,3,2,4}, for instance, in \cite{1} Friedman and McLeod have studied the zero Dirichlet problem of the semilinear heat equation: \begin{equation}\label{shahed} u_t=\Delta u+f(u),\quad \mbox{in}\quad B_R\times (0,T),\end{equation} under fairly general assumptions on $u_0$ (nonincreasing radial function, vanishing on $\partial B_R$). They have considered the two special cases:
\begin{eqnarray}\label{fort} u_t&=&\Delta u+u^p,\quad p>1, \\
\label{repair} u_t&=&\Delta u+e^u. \end{eqnarray}

For equation (\ref{fort}),  they showed that for any $\alpha >2/(p-1),$ the upper pointwise estimate takes the following form
$$u(x,t)\le C|x|^{-\alpha}, \quad x \in B_R\setminus\{0\}\times (0,T),$$
which shows that the only possible blow-up point is $x=0.$  Moreover, under an additional assumption of monotonicity in time (\ref{6}), the corresponding lower estimate on the blow-up can be established (see\cite{2}) as follows
$$u(x,T)\ge C|x|^{-2/(p-1)},\quad x \in B_{R^*}\setminus\{0\},$$
  for some $R^*\le R,$ $C>0.$ On the other hand, it has been shown in \cite{1} that the upper (lower) blow-up rate estimates take the following form
$$c(T-t)^{-1/(p-1)}\le u(0,t)\le C(T-t)^{-1/(p-1)},\quad t \in  (0,T).$$
For the second case, (\ref{repair}), Friedman and McLeod showed similar results, they proved that the point $x=0$ is the only blow-up point, and that due to the upper pointwise estimate, which takes the following form
$$u(x,t)\le \log C+\frac{2}{ \alpha}\log(\frac{1}{x}),\quad (x,t) \in B_R\setminus\{0\}\times (0,T),$$ where  $0< \alpha<1, C>0.$ 
Moreover, the upper (lower) blow-up rate estimate takes the following form
\begin{equation}\label{Zam} \log c-\log(T-t) \le u(0,t)\le \log C-\log(T-t),\quad t \in  (0,T).\end{equation}

The aim of section two, is to show that the results of Friedman and McLeod hold true for problem (\ref{1}). On other words, we prove that $x=0$ the only possible blow-up point for this problem. Furthermore, we show that the upper blow-up rate estimate takes the following form
$$u(0,t) \le \log C-\frac{1}{p} \log(T-t), \quad t\in (0,T).$$

The problem of the heat equation defined in a ball $B_R$ with a nonlinear Neumann boundary condition, $\frac{\partial u}{\partial \eta}=f(u)$ on $\partial B_R\times (0,T),$ has been introduced in \cite{17,15,16,14}, for instance, in \cite{14}  it has been shown that if $f$ is nondecreasing and $1/f$ is integrable at infinity for $u>0$, then the blow-up occurs  in finite time for any positive initial data $u_0$ (not necessarily radial), moreover, if $f$ is $C^2(0,\infty),$ increasing and  convex  in $(0,\infty),$ then blow-up occurs only on the boundary. 

For the special case, where $f(u)=u^p,$ it has been proved in \cite{15} that for any $u_0,$ the finite time blow-up occurs where $p>1,$ and it occurs only on the boundary. Moreover, it has been shown in \cite{16,22} that the upper (lower) blow-up rate estimate take the following form
$$C_1(T-t)^{\frac{-1}{2(p-1)}}\le\max_{x\in \overline{B}_R}u(x,t)\le C_2(T-t)^{\frac{-1}{2(p-1)}}, \quad t\in (0,T).$$

In \cite{17}, it has been considered the second special case, where, $f(u)=e^u,$ in one dimensional space defined in the domain $(0,1)\times(0,T),$ it has been proved that every positive solution blows up in finite time and the blow-up occurs only on the boundary ($x=1$) and the upper (lower) blow-up rate estimates take the following forms
$$C_1(T-t)^{-1/2}\le e^{u(1,t)}\le C_2(T-t)^{-1/2}, \quad 0<t<T.$$

Section three concerned with the blow-up solutions of  problem (\ref{b1}), we prove that the upper blow-up rate estimate takes the following form $$\max_{\overline{B}_R}u(x,t) \le \log C -\frac{1}{2p}\log(T-t), \quad  0<t<T.$$

\section{Problem (\ref{1})}
\subsection{Preliminaries}
  Since $f(u)=e^{u^p}$ is $C^1([0,\infty))$ function, the existence and uniqueness of local classical solutions to problem (\ref{1}) are well known, see \cite{23,37}. On the other hand,  since the function $f$ is convex on $(0,\infty)$ and satisfies the condition (\ref{T}), therefore, the solutions of problem (\ref{1}) blow up in finite time for large initial function. 

The next lemma shows some properties of the solutions of problem (\ref{1}). We denote for simplicity $u(r,t)=u(x,t).$
  \begin{lemma}\label{Wk}
    Let $u$ be a classical solution of (\ref{1}). Then
\begin{enumerate}[\rm(i)]
 \item $u(x,t)$ is positive and radial, $u_r\le 0$  in $[0,R) \times (0,T).$ Moreover, $u_r<0$  in $(0,R]\times (0,T).$
\item $u_t > 0 ,\quad  (x,t) \in B_R\times (0,T).$
\end{enumerate}
\end{lemma}

\subsection{Pointwise Estimates}

This subsection considers the pointwise estimate to the solutions of problem (\ref{1}),  which shows that the blow-up cannot occur if $x$ is not equal zero. In order to prove that, we need first to recall the following lemma, which has been proved by Friedman and McLeod in \cite{1}.

\begin{lemma}
Let $u$ be a blow-up solution of the zero Dirichlet problem of (\ref{shahed}) with $u_0$ is nonzero, nonincreasing radial function vanishing on $\partial B_R.$  Also suppose that 
\begin{equation}\label{Frid} u_{0r}(r)\le -\delta r,\quad\mbox~ for ~0<r\le R,\quad\mbox{where}~\delta>0.\end{equation}
If there exist $F\in C^2(0,\infty)\cap C^{1}([0,\infty)),$ such that $F$ is positive in $(0,\infty)$ and satisfies  \begin{equation}\label{mar} \int^\infty_s \frac{du}{F(u)}< \infty,\quad F^{'},F^{''}\ge 0\quad \mbox{in}~(0,\infty). \end{equation} Also if it satisfies with $f$ the following condition,
\begin{equation}\label{maan}
f^{'}F-fF^{'}\ge 2\varepsilon FF^{'} \quad \mbox{in} \quad (0,\infty),
\end{equation} then the function
$J=r^{n-1}u_r+\varepsilon r^nF(u)$ is nonpositive  in $B_R\times (0,T)$  for some $\varepsilon >0.$
\end{lemma}

\begin{theorem}\label{BN}
Let $u$ be a blow-up solution of problem (\ref{1}). Also suppose that $u_0$ satisfies (\ref{Frid}). Then $x=0$ is the only blow-up point. \end{theorem}

\begin{proof}
  Let  $$F(u)=e^{ \delta u^p},\quad 0<\delta <1.$$ It is clear that $F$ satisfies (\ref{mar}).
The next aim is to show that the inequality (\ref{maan}) holds.

A direct calculation shows 
\begin{eqnarray} \label{ead}  f^{'}(u)F(u)-f(u)F^{'}(u)&=&pu^{p-1}e^{(1+\delta)u^p}-\delta pu^{p-1} e^{(1+\delta)u^p} \\
&=&pu^{p-1}e^{(1+\delta)u^p}[1-\delta] \nonumber.\end{eqnarray} 
On the other hand,\begin{equation}\label{eaad} 2\varepsilon F(u)F^{'}(u)=2\varepsilon \delta pu^{p-1}e^{2\delta u^p}. \end{equation}
From (\ref{ead}), (\ref{eaad}) it is clear that (\ref{maan}) holds true
provided $\varepsilon ,\delta$ are small enough.

Thus$$J=r^{n-1}u_r+\varepsilon r^n e^{ \delta u^p} \le 0,\quad (r,t)\in (0,R)\times (0,T),$$
or \begin{equation}\label{yah}-\frac{u_r}{e^{ \delta u^p}} \ge \varepsilon r. \end{equation}
Let $G(s)=\int^\infty_s \frac{du}{ e^{ \delta u^p}}.$

It is clear that
$$\frac{d}{dr}G(u(r,t))=\frac{d}{dr}\int_u^\infty\frac{du}{e^{ \delta u^p}}=-\frac{d}{dr}\int_\infty^u\frac{du}{e^{ \delta u^p}}=-\frac{d}{du}\int_\infty^u\frac{u_r}{e^{ \delta u^p}}du=-\frac{u_r}{e^{ \delta u^p}}.$$

Thus, by (\ref{yah}), we obtain $$G(u(r,t))_r\ge \varepsilon r.$$
Now, integrate the last equation from $0$ to $r$
$$G(u(r,t))-G(u(0,t))\ge \frac{1}{2}\varepsilon r^2.$$
It follows
\begin{equation}\label{faw}
G(u(r,t))\ge \frac{1}{2}\varepsilon r^2.
\end{equation}
If for some $r>0,$ $u(r,t)\rightarrow \infty,$  as $t\rightarrow T,$ then $G(u(r,t))\rightarrow 0,$  as $t\rightarrow T,$ a contradiction to (\ref{faw}).
\end{proof}

\subsection{Blow-up Rate Estimate}
  The following theorem considers the upper bounds of the blow-up rate for problem (\ref{1}).
  \begin{theorem}\label{see}
  Let $u$  be a solution of (\ref{1}), which blows up at only $x=0,$ in finite time $T.$ Then there exists a positive constant $C$  such that
\begin{equation}\label{Sf} u(0,t) \le \log C-\frac{1}{p} \log(T-t), \quad t\in (0,T).\end{equation}
\end{theorem}
\begin{proof}
Define the  function $F$ as follows,
$$F(x,t)=u_t-\alpha f(u), \quad (x,t) \in B_ R \times (0,T),$$ where $f(u)=e^{u^p},~\alpha >0.$

A direct calculation shows
\begin{eqnarray*} F_t-\Delta F&=&u_{tt}-\alpha f^{'}u_t-\Delta u_t + \alpha \Delta f(u),\\
&=&u_{tt}-\Delta u_t -\alpha f^{'}[u_t-\Delta u]+\alpha {|\nabla u|}^2 f^{''},\\
&=&f^{'}u_t-\alpha f^{'}f(u)+\alpha {|\nabla u|}^2 f^{''}.\end{eqnarray*} Thus
\begin{equation}\label{13}
F_t-\Delta F-f^{'}(u)F=\alpha {|\nabla u|}^2 f^{''} \ge 0, \quad (x,t) \in B_R \times (0,T),
\end{equation}  due to $f^{''}(u)>0,$ for $u$ in $(0,\infty).$ 

Since, $f^{'}$ is continuous, therefore, $f^{'}(u)$ is bounded in $\overline{B}_R\times [0,t],$ for $t<T.$

By Lemma \ref{Wk}, $u_t(x,t) >0,~ \mbox {in}~B_R \times (0,T),$ and since $u$ blows up at $x=0,$ therefore, there exist $ k>0,$ $\varepsilon \in (0,R),$ $\tau\in (0,T)$ such that
$$u_t(x,t)\ge k,\quad (x,t)\in \overline B_\varepsilon \times [\tau,T).$$
Also, we can find $\alpha >0$ such that $u_t(x,\tau)\ge \alpha f(u(x,\tau)),$ for $x \in  B_\varepsilon.$ 
Thus \begin{equation}\label{14} F(x,\tau)\ge 0 \quad\mbox{for}~ x \in  B_\varepsilon.\end{equation}

 On the other hand, because of $u$ blows up at only $x=0,$ there exists $C_0>0$ such that  $$f(u(x,t)) \le C_0 < \infty, \quad \mbox{in}\quad \partial{B}_\varepsilon\times (0,T),$$ 
If we choose $\alpha$ is small enough such that $k\ge \alpha C_0,$ then we get
\begin{equation}\label{15}
F(x,t)\ge 0, \quad (x,t) \in \partial B_\varepsilon \times [\tau, T), \end{equation} 

By (\ref{13}), (\ref{14}), (\ref{15}) and maximum principle \cite{21}, it follows that 
$$F(x,t)\ge 0, \quad (x,t) \in \overline{B}_\varepsilon \times (\tau,T).$$
Thus \begin{equation}\label{16}
u_t(0,t) \ge \alpha e^{u^p(0,t)},  \quad \mbox{for}\quad \tau\le t < T.
\end{equation}
Since $u$ is increasing in time and blows at $T,$ there exist $\tau^*\le \tau$ such that
 $$u(0,t)\ge p^{\frac{1}{(p-1)}}\quad \mbox{for}\quad \tau^* \le t< T,$$ 
provided $\tau$ is close enough to $T,$ which leads to 
\begin{equation}\label{17}
e^{u^p(0,t)}\ge e^{pu(0,t)},\quad \tau^*\le t< T.
\end{equation}
From (\ref{16}), (\ref{17}), it follows that
\begin{equation}\label{19}
u_t(0,t) \ge \alpha e^{pu(0,t)},  \quad \mbox{for}\quad \tau\le t< T.
\end{equation}
Integrate (\ref{19}) from $t$ to $T$
\begin{equation*}\label{20}
 \int_t^{T} u_t(0,t)e^{-pu(0,t)} \ge \alpha (T-t).
\end{equation*}
Thus
\begin{equation}\label{21}
-\frac{1}{p}e^{-pu(0,t)} |^T_t \ge \alpha(T-t).
\end{equation}
Since
\begin{equation*}\label{22}
u(0,t) \rightarrow \infty, \quad e^{-pu(0,t)}  \rightarrow 0,\quad \mbox{as} \quad t \rightarrow T,
\end{equation*}
therefore, (\ref{21}) becomes
\begin{equation*}\label{23}
\frac{1}{ e^{pu(0,t)}}\ge p\alpha (T-t).
\end{equation*}
Thus
$$e^{pu(0,t)}(T-t) \le C^*, \quad C^*=1/(p\alpha),\quad t\in [\tau,T)$$
Therefore, there exist a positive constant $C$ such that 
$$u(0,t) \le \log C-\frac{1}{p} \log(T-t), \quad t\in (0,T).$$
\end{proof}

\section{Problem (\ref{b1})}

\subsection{Preliminaries}
The local existence of the unique classical solutions to problem (\ref{b1}) is well known (see \cite{14}). On the other hand, since $f(u)=e^{u^p}$ is $C^2(0,\infty),$ increasing, positive function in $(0,\infty)$ and $1/f$ is integrable at infinity for $u>0$, moreover, $f$ is convex ($f^{''}(u)> 0,\forall u> 0$). Therefore, according to the result of \cite{14}, the solutions of problem (\ref{b1}) blow up infinite time and the blow-up occurs only on the boundary.

 The following lemma shows some properties of the solutions of problem (\ref{b1}). We denote for simplicity $u(r,t)=u(x,t).$
  \begin{lemma}\label{asda}
    Let $u$ be a classical unique solution to problem (\ref{b1}).Then
  \begin{enumerate}[\rm(i)]
\item $u> 0,$  radial  on  $\overline{B}_R\times (0,T).$ Moreover, 
$u_r \ge 0,$ in $[0,R]\times [0,T).$
 \item   $u_t>0$ in $\overline B_R\times (0,T).$ 
 Moreover, if $\Delta u_0\ge a>0,$ in $\overline B_R,$ then $u_t \ge a,$ in $\overline{B}_R\times [0,T).$
  \end{enumerate}
\end{lemma}

\subsection{Blow-up Rate Estimate }
The following theorem considers the upper blow-up rate estimate of problem (\ref{b1}).
\begin{theorem}\label{ser}
Let u be a blow-up solution to (\ref{b1}), where $\Delta u_0\ge a>0$ in $\overline B_R,$ $T$ is the blow-up  time.Then there exists a positive constant $C$ such that
\begin{equation}\label{hn} \max_{\overline{B}_R}u(x,t) \le \log C -\frac{1}{2p}\log(T-t), \quad  0<t<T.\end{equation} 
\end{theorem}
\begin{proof}
We follow the idea of \cite{17}, consider the function $$F(x,t)=u_t(r,t)-\varepsilon u^2_r(r,t),\quad  (x,t) \in B_R \times (0,T).$$

By a straightforward calculation 
$$F_t-\Delta F=2 \varepsilon (\frac{n-1}{r^2}u^2_r+u^2_{rr})\ge 0.$$

Since $\Delta u_0\ge a>0,$ and $u_{0r}\in C(\overline B_R),$ 
$$F(x,0)=\Delta u_0(r)-\varepsilon u_{0r}^2(r)\ge 0, \quad x \in B_R.$$ 
provided $\varepsilon$ is small enough. 

Moreover,
\begin{eqnarray*}\frac{\partial F}{\partial \eta}|_{x \in S_R}&=&u_{rt}(R,t)-2\varepsilon u_r(R,t)u_{rr}(R,t)\\ &=&(e^{u^p(R,t)})_t-2\varepsilon e^{u^p(R,t)}(u_t(R,t)-\frac{n-1}{r}u_r(R,t)) \\&\ge& (p[u(R,t)]^{p-1}-2\varepsilon)e^{u^p(R,t)} u_t(R,t).\end{eqnarray*}
Since $$u_t>0, \quad \mbox{on}\quad \overline{B}_R \times (0,T).$$
Thus  $$\frac{\partial F}{\partial \eta}|_{x\in S_R}\ge 0,\quad t \in (0,T)$$ provided $$\varepsilon \le \frac{p[u_0(R)]^{p-1}}{2}.$$ 
From the comparison principle \cite{21}, it follows that
 $$F(x,t)\ge 0,\quad \mbox{in} \quad \overline{B}_R \times ( 0,T),$$ in particular
$F(x,t) \ge 0,$ for $|x|=R,$ that is
$$u_t(R,t) \ge \varepsilon u^2_r(R,t)=\varepsilon e^{2u^p(R,t) },\quad t\in (0,T).$$  

Since $u$ is increasing in time and blows at $T,$ there exist $\tau\le T$ such that
 $$u(R,t)\ge p^{\frac{1}{(p-1)}}\quad \mbox{for}\quad\tau\le t< T,$$ which leads to $$u_t(R,t) \ge\varepsilon e^{2pu(R,t) },\quad t \in [\tau,T).$$

By integration the above inequality from $t$ to $T,$ it follows that
\begin{equation*}
\int_t^{T} u_te^{-2pu(R,t)} \ge \varepsilon (T-t).
\end{equation*}
So
\begin{equation}\label{b8}
-\frac{1}{2p}e^{-2pu(R,t)}|^T_t \ge \varepsilon(T-t).
\end{equation}
Since
\begin{equation*}\label{b9}
u(R,t) \rightarrow \infty ,\quad e^{-pu(R,t)}\rightarrow 0 \quad \mbox{as } \quad t \rightarrow T,
\end{equation*}
 the inequality (\ref{b8}) becomes 
\begin{equation*}
\frac{1}{ e^{pu(R,t)}}\ge (2p\varepsilon (T-t))^{1/2},
\end{equation*} which means 
$$(T-t)^{1/2}e^{pu(R,t)} \le \frac{1}{\sqrt {2p\varepsilon}},$$
Therefore, there exist a positive constant $C$ such that  
$$\max_{\overline{B}_R}u(x,t) \le \log C -\frac{1}{2p}\log(T-t), \quad  0< t<T.$$   
\end{proof}

\end{document}